\documentclass[12pt]{amsart}
\usepackage[english]{babel}
\usepackage{color}
\usepackage{url}
\newcommand{\del}{\partial}
\newtheorem{te}{Theorem}[section]
\newtheorem{pr}{Proposition}[section]

\newtheorem{definition}{Definition}[section]
\newtheorem{remark}{Remark}[section]

\begin{document}

\title[The pluriclosed flow for $T^2$-invariant  Vaisman metrics]{The pluriclosed flow for $T^2$-invariant  Vaisman metrics on the Kodaira-Thurston surface}
\author{Anna Fino}
\address[Anna Fino]{Dipartimento di Matematica ``G. Peano'', Universit\`{a} degli studi di Torino \\
Via Carlo Alberto 10\\
10123 Torino, Italy\\
\& Department of Mathematics and Statistics, Florida International University\\
Miami, FL 33199, United States}
\email{annamaria.fino@unito.it, afino@fiu.edu}

\author{Gueo Grantcharov}
\address[Gueo Grantcharov]{Department of Mathematics and Statistics \\
Florida International University\\
Miami, FL 33199, United States}
\email{grantchg@fiu.edu}

\author{Eddy Perez}
\address[Eddy Perez]{Department of Mathematics and Statistics \\
Florida International University\\
Miami, FL 33199, United States}
\email{epere389@fiu.edu}



\maketitle

\begin{abstract} In this note we study  $T^2$-invariant  pluriclosed metrics on the Kodaira-Thurston surface. We  obtain a characterization  of  $T^2$-invariant Vaisman  metrics, and notice that the Kodaira-Thurston surface admits  Vaisman metrics with non-constant scalar  curvature. Then we study the behaviour  of the Vaisman condition  in relation to the pluriclosed flow.
As a consequence, we show that  if the initial metric  on the Kodaira-Thurston surface is a $T^2$-invariant  Vaisman metric, then the pluriclosed flow preserves the Vaisman condition, extending to the non-constant scalar curvature case the previous  result  in  \cite{FT}.

\end{abstract}

\section{introduction}

Given a Hermitian manifold $(M,J,  g)$ of complex dimension $n$ the Bismut connection $\nabla^B$  (also known as the Strominger connection)  is the unique connection on $M$ that is Hermitian  (i.e.  $ \nabla^B J =0$, $\nabla^B g =0$)  and has totally skew-symmetric torsion tensor. Its explicit expression appeared  in Strominger's paper \cite{Strominger} and independently in Bismut’s paper \cite{Bismut}, where $\nabla^B$ was used  in to study  local index theorems.

A  Hermitian metric $g$  on a  complex  manifold  is called {\em pluriclosed} if its  fundamental form $\omega$ satisfies $\partial \overline \partial \omega =0$ or equivalently  if the  Bismut  torsion $3$-form is closed.
  This is a weaker restriction than the K\"ahler condition on $g$ and every compact complex surface admits pluriclosed metrics  \cite{gauduchon}.
  The pluriclosed metrics have been studied by many authors. We refer for instance  to  \cite{AI, FinoTom, ZZ, YZZ} and the references therein for more  backgrounds and results.

  In  \cite{streets-tian} Streets and Tian  introduced a parabolic flow of pluriclosed metrics, also  called the {\em  pluriclosed flow},    which is defined by the equation
$$
\frac{\del}{\del\,t}\omega(t)= - {(\rho^B)}^{1,1}\,,\qquad
\omega(0)=\omega_0,
$$
where ${({\rho^B})^{1,1}:= {(\rho^B (\omega(t))}^{1,1}}$ denotes the $(1,1)$-part of the Ricci form of the Bismut connection and $\omega_0$ is  the fundamental form of a fixed   pluriclosed  metric. This flow preserves the pluriclosed  condition and    also the existence of generalized K\"ahler structures  \cite{ST2}.

A classification of compact solitons for the pluriclosed flow on complex surfaces was given in \cite{Streets3}, showing  that the complex surface underlying a soliton must be either  K\"ahler or  a  Hopf surface.  On   class 1 Hopf surfaces there exists a non-trivial pluriclosed steady soliton,  but the question of existence on  class 0 Hopf surfaces  is open.

It was conjectured in \cite{AOUV}  that if a compact Hermitian manifold  $(M, J, g)$, is with K\"ahler-like Bismut connection, i.e. if  $\nabla^B$ satisfies the first Bianchi identity and the  condition
$$
R^B (X,Y,JZ,JW) = R^B(X,Y,Z,W) = R^B(JX,JY,Z,W)
$$
for every tangent vector fields $X, Y, Z$ and $W$, then $g$ has to pluriclosed.
An affirmative answer to the above conjecture,  even without the compactness assumption, has been given in \cite{ZZ}, showing  that  $\nabla^B$ is K\"ahler-like if and only if  $\nabla^B$ has parallel torsion and $g$ is pluriclosed.
As proved in \cite{AOUV},  already in complex dimension $2$  there are examples of compact Hermitian manifolds which are Bismut K\"ahler-like, but not Bismut flat. The simplest such example is  given by the Kodaira-Thurston surface \cite{Kodaira,Thurston}.

For complex surfaces  the Bismut  K\"ahler-like condition is equivalent to the Vaisman condition. Recall that a  {\em Vaisman metric } $g$ on a complex manifold $(M, J)$ is a Hermitian metric  whose fundamental form $\omega$ satisfies $d \omega  = \theta \wedge \omega$  for some  1-form $\theta$ which is  $d$-closed and parallel with respect  to the Levi-Civita connection. When the complex surface has odd first  Betti number,   Belgun \cite{belgun-surfaces} gave a complete classification of all such metrics. In particular,  the complex surface  is either a properly elliptic surface, a Kodaira surface or an elliptic or Class 1 Hopf surface. If  the complex surface  has even first Betti number, then  the Vaisman metric  must be   K\"ahler and hence   there exist  no non-K\"ahler Vaisman metrics on such surfaces. N. Istrati \cite[ Example 7.1]{Istrati}  noticed that every compact complex surface admitting Vaisman metric has vanishing real first Chern class.


 In \cite{FT} it has been  shown that  the Vaisman condition  on complex surfaces is preserved along the pluriclosed flow if  the initial metric has constant scalar curvature. One motivation for the present paper is to  check whether the constant scalar curvature condition is necessary.  A complete description of the global existence and convergence for the Ricci-Yang-Mills flow on torus  bundles over Riemann surfaces have been obtained in \cite{Streets}. So a natural question is to investigate if the Vaisman condition is preserved  for  torus bundles over Riemann surfaces.

 We focus on Kodaira-Thurston surface which can be also viewed  as total space of a  principal  torus bundle over the $2$-torus $T^2$. We first  obtain a description of $T^2$-invariant  Vaisman metrics,  showing that  there are $T^2$-invariant  Vaisman  metrics which have non-constant scalar curvature.  Then,  as  a main result we prove that if  $\omega_0$  is the fundamental form of  a  $T^2$-invariant  Vaisman metric on the Kodaira-Thurston surface then the pluriclosed flow starting with $\omega_0$  preserves the Vaisman condition.

The content of the paper is as follows: In Section 2 we collect the main known results about the Vaisman manifolds which we need later. Then in Section 3 we prove the main result.

\vspace{.5in}

\section{Preliminaries}

Let $(M,J)$ be a complex manifold of complex dimension $n$ and let $g$ be a Hermitian metric on $X$ with associated fundamental form $\omega(\cdot\,,\cdot)=g(\cdot,  J \cdot)$\,. An affine connection is called Hermitian if it preserves the metric $g$ and the complex structure $J$. In particular, Gauduchon in \cite{gauduchon-bumi} proved that there exists an affine line $\left\lbrace\nabla^t\right\rbrace_{t\in\mathbb{R}}$ of canonical Hermitian connections, passing through the \emph{Chern connection} and the \emph{Bismut connection}; these connections are completely determined by their torsion.
Let $\nabla$ be a Hermitian connection and $T(X,Y)=\nabla_X Y-\nabla_Y X -[X,Y]$ be its torsion, we denote with the same symbol
$$
T(X,Y,Z):=g(T(X,Y), Z).
$$
Then the Chern connection $\nabla^{Ch}$ is the unique Hermitian connection whose torsion has trivial $(1,1)$-component and the Bismut connection (also called Strominger connection) $\nabla^B$ is the unique Hermitian connection with totally skew-symmetric torsion. In particular, the torsion of the Bismut connection satisfies
$$
T^B(X,Y,Z)=d^c\omega(X,Y,Z), 
$$
where $d^c=-J^{-1}dJ$.

The Chern and Bismut connections are related to the Levi-Civita connection $\nabla^{LC}$ by
$$
\begin{aligned}
g(\nabla^B_ X Y, Z)&=g(\nabla^{LC}_X Y, Z)+\frac{1}{2}d^c\omega(X,Y,Z)\,,\\
g(\nabla^{Ch}_X Y, Z)&=g(\nabla^{LC}_X Y,Z)+\frac{1}{2}d\omega(JX,Y,Z)\,.
\end{aligned}
$$

A Hermitian metric $g$ on a complex manifold $(M, J)$  is called  \emph{pluriclosed}  or  \emph{strong K\"ahler with torsion} (\emph{SKT} for brevity) if $T^B$ is a closed $3$-form, namely $dT^B=0$, or equivalently $dd^c\omega=0$.\\
Recall that the trace of the torsion of the Chern connection is equal to the Lee form of $g$ (cf. \cite{gauduchon}), that is the $1$-form defined by
$$
\theta=Jd^*\omega,
$$
where $d^*$ is the adjoint of the exterior derivative $d$ with respect to $g$, or equivalently $\theta$ is the unique $1$-form satisfying
$$
d\omega^{n-1}=\theta\wedge\omega^{n-1}\,.
$$
A Hermitian metric $g$ is called \emph{Gauduchon} if
$dd^c\omega^{n-1}=0$, or equivalently
$d^*\theta=0$. In particular, in  complex dimension $2$ Gauduchon and pluriclosed  metrics coincide.

We recall the following
\begin{definition}
A Hermitian metric $g$ on  a complex manifold $(M, J)$ is called \emph{locally conformally K\"ahler} ({\emph{LCK}} for brevity) if
$$
d\omega=\alpha\wedge\omega,
$$
where $\alpha$ is a $d$-closed $1$-form. In particular, $\alpha=\frac{1}{n-1}\theta$ and $\theta$ is $d$-closed.\\
A locally conformally K\"ahler metric $g$  is called \emph{Vaisman} if the Lee form is parallel with respect to the Levi-Civita connection $\nabla^{LC}$, namely
$$
\nabla^{LC}\theta=0\,.
$$
\end{definition}
An immediate consequence is that the Vaisman metrics are  Gauduchon and the norm of the Lee form $|\theta|$ with respect to them is constant.\\
So on complex surfaces Vaisman metrics are pluriclosed and $
T^B=-*\theta.
$

A Vaisman structure on a complex manifold is uniquely determined (up to a positive constant) by its Lee form $\theta$ via the following
$$
\omega=\frac{1}{|\theta|^2}(\theta\wedge J\theta-dJ\theta).
$$
Moreover, by \cite{OV}  the dual Lee vector field $T= \theta^\#$   is holomorphic and Killing.
A Vaisman metric is called normalized if the Lee form (or, equivalently, the Lee vector field) has norm $1$. Moreover, for a Vaisman structure
$$
d(J \theta) = \theta \wedge J \theta - |\theta|^2 \omega
$$
is always of type $(1,1)$.
By \cite{MMO}   a  pluriclosed  locally conformally K\"ahler  metric  on a compact complex manifold with holomorphic Lee vector field is Vaisman.

If  $(M, J, g, \omega, \theta)$  is a  normalized Vaisman manifold, then for every positive real number $b > 0$ and a harmonic 1-form $\alpha$, pointwise orthogonal to $\theta$ and $J \theta$,  the pair $(\tilde \omega,  \tilde \theta)$ defined by
$$
\tilde \theta :=b  \, \theta  + \alpha, \quad   \tilde \omega :=  \tilde \theta   \wedge J \tilde \theta  - dJ \tilde \theta$$  is a normalized Vaisman structure on $(M, J)$ (see for instance Lemma 3.2 in \cite{MMP}).  The Vaisman structure $(\tilde \omega, \tilde \theta)$ is called  a deformation of type I of  $(\omega, \theta)$.

There is another type of deformation of Vaisman structures, which preserves the cohomology class of the Lee form. Let $(g, \omega, \theta)$ be a Vaisman structure on $(M,J)$ with Lee vector field $T$. Let  $f \in {\mathcal C}^{\infty} (M)$  such that  $T(f) = JT(f) = 0.$ If one defines the closed $1$-form $\tilde \theta$ and the $(1,1)$-form $\tilde \omega$  by
$$
\tilde \theta = \theta + df, \quad \tilde \omega = |\theta|_g^2 \omega + \theta \wedge J df + df \wedge J \theta + df \wedge J df - d d^c f,
$$
then when $\tilde \omega$ is positive, the structure $(\tilde \omega, \tilde \theta)$ is a
normalized Vaisman  structure, called deformation of type II of $(\omega, \theta)$. The Lee vector field  $\tilde T$ of  $(\tilde \omega, \tilde \theta)$ is given by $\frac{1}{|\theta|^2_g} T$.

By Proposition 3.7 in \cite{MMP} if  $(M, J,g,  \omega, \theta_0)$  be a compact Vaisman manifold. Then any normalized Vaisman structure $(\tilde \omega, \tilde \theta)$ on $(M, J)$  is obtained by deformations of type I and II starting from the given Vaisman structure  $(\omega, \theta)$.

\begin{remark}  By   formula (2.7) in \cite{AI}  and \cite{gauduchon-bumi} on  a complex surface the Bismut and Chern Ricci forms are related by the following relation
\begin{equation}\label{relRicciforms}
\rho^{Ch}  =\rho^{B} + d (J \theta).
\end{equation}

By \cite{FT}  if  $(M, J, \omega) $  is  a compact  Vaisman  surface  with Lee form $\theta$, then
$
\rho^{Ch}= h \,dJ\theta
$
for some $h\in\mathcal{C}^\infty(M,\mathbb{R})$.
Moreover, the scalar curvature of $\omega$ is constant if and only if $h$ is constant. Clearly when $h$ is constant, $c_1(M)=0$, but in \cite{Istrati} it was noted that every compact Vaisman surface has vanishing real first Chern class.

Note that, given a normalized Vaisman structure $(\omega, \theta)$ on a complex surface,  if we  consider an  orthonormal basis $(\theta,  J \theta, \xi, J \xi)$ then
$$
\omega = \xi \wedge J \xi + \theta \wedge J \theta, \quad \rho^B = (h - 1) \xi \wedge J \xi = (h -1) d J \theta.
$$
Therefore $h -1 =  \rho^B (\xi ^\#,  J \xi^\#)$.

We note also that by \cite{Istrati} the Vaisman manifolds with vanishig (real) first Chern class fall into 3 classes, depending on the sign  of the first Bott-Chern class. In what follows we focus on the Kodaira-Thurston surface which has  $c_1^{BC} = 0$.

\end{remark}

\section{$T^2$-invariant pluriclosed metrics on the Kodaira-Thurston surface}

The \emph{ Kodaira-Thurston surface} is defined as the compact $4$-manifold
\begin{equation*}
M=Nil^3/\Gamma\times S^1,
\end{equation*}
 where
$Nil^3$ is  the $3$-dimensional  real Heisenberg group
\begin{equation*}
Nil^3=\left\{\left[\begin{smallmatrix}1&x&z\\0&1&y\\0&0&1\end{smallmatrix}\right]\mid x,y,z \in \mathbb R\right\},
\end{equation*}
and $\Gamma$ is  the lattice in $Nil^3$ of matrices having integers entries.

Therefore $M$ is parallelizable  and has a global left-invariant co-frame
$$
e^1= d y,\quad e^2= dx,\quad e^3= dw\quad e^4= dz-x dy
$$
satisfying the structure equations
$$
 de^1= de^2= de^3=0,\quad  de^4=e^{12},
$$
with
\begin{equation*}
e^{ij}=e^i\wedge e^j.
\end{equation*}
The   dual  left-invariant  frame  is given by
$$
e_1 = \partial_y + x \partial_z, \,  e_2 =  {\partial_x}, \,  \, e_3 = \partial_w, \,  e_4 = \partial_z.
$$

Every smooth map $u\colon M\to \mathbb R$ can  be regarded as a smooth map $u\colon \mathbb R^4\to \mathbb R$ satisfying the
periodicity condition
\begin{equation*}
u(x+j,y+k,z+jy+m,w+n)=u(x,y,z,w),
\end{equation*}
for all $(x,y,z,t)\in \mathbb R^4$ and $(j,k,m,n)\in\mathbb Z^4$.
We consider on $M$  the complex structure given by
$$
Je^1=e^2,\qquad Je^3=e^4.
$$
A global frame of $(1,0)$-forms is given by
$$
\varphi^1 = e^1 +  i e^2 = dy + i dx, \quad \varphi^2 = e^3 + i e^4 =  dw+ i(dz-x dy).
$$
Therefore
$$
d \varphi^1 =0,  \quad d \varphi^2 = - \frac 12 \varphi^{1\overline 1}.
$$
Moreover,  $ y + i x$ and $w + i z - \frac{1}{2} x^2$ are  local holomorphic coordinates.

Since $Nil^3/\Gamma\times S^1=(Nil^3\times \mathbb R)/(\Gamma\times \mathbb Z)$,  the Kodaira-Thurston surface $M$
is a $2$-step nilmanifold and every left-invariant  Hermitian  structure on $Nil^3 \times \mathbb R$ projects to
a Hermitian structure  on $M$. Moreover, the compact $3$-dimensional manifold $N=Nil^3/\Gamma$ is the total
space of  an $S^1$-bundle over a $2$-dimensional torus $T^2$ with projection $\pi_{xy}\colon N\to T^2_{xy}$ and $M$
inherits a structure of  principal $T^2$-bundle over the $2$-dimensional torus $T^2_{xy}$. Then it makes sense to consider differential forms  invariant by the action of the fiber  $T^2_{tz}$. A $k$-form $\phi$ on $M$ is invariant by the action of the fiber $T^2_{zt}$ if  its
coefficients with respect to the global basis $e^{j_1}\wedge \cdots \wedge e^{j_k}$ do not depend on
the variables $z,t$.

An  arbitrary  $T^2$-invariant $J$-invariant  metric  $g$ on $M$  has associated fundamental form
$$
\omega = \frac{1}{2}  i  r \,  \varphi^{1 \overline 1} + \frac{1}{2}  i s \,  \varphi^{2 \overline 2}  + \frac{1}{2} (u \varphi^{1 \overline 2} -  \overline u \varphi^{2 \overline1}) = r e^{12} + s e^{34} + u_1 (e^{13} + e^{24}) + u_2 (e^{14} - e^{23}),
$$
where $r,s$ are real functions  $r  = r(x,y), s= s(x,y)$ and  $u = u (x,y)$ is a complex values   function such that  $r (x, y) >0$, $s (x, y) >0$, $r (x,y)  \, s(x, y) > |u(x,y)|^2$, i.e.  the  Hermitian  matrix

$$
H =  \frac{1}{2}  \left ( \begin{array}{cc} r & - i u\\[2pt]   i \overline u& s \end{array} \right )
$$

is positive definite. In particular, if $u = u_1 + i u_2$, we have
$$
\begin{array}{l}
g (e_1, e_1) = g(e_2, e_2) = r,  \quad g (e_3, e_3) = g(e_4, e_4) = s, \\[2pt]
g (e_1, e_3) = u_2 = g(e_2, e_4),  \quad g (e_1, e_4) = - u_1 = -g(e_2, e_3).
\end{array}
$$
Note that $\omega$ is left-invariant if and only if $r, s$ and $u$ are constant functions.


Moreover
$$d  r = r_x e^2 + r_y e^ 1 =   -\frac 12  i r_x (\varphi^1- \overline \varphi^1) +\frac 12 r_y  (\varphi^1 +\overline  \varphi^1)   $$
and a similar relation holds for $d s$, $du$ and $d \overline u$.
As a consequence
$$
\partial r =   \left(\frac 12 r_y - \frac 12 i r_x \right ) \varphi^1, \quad   \overline \partial r =  \left(\frac 12 r_y +\frac 12 i r_x \right ) \overline  \varphi^1.
$$
Therefore, one has
$$
d \omega =  \frac{1}{4}  (s_x+i s_y) \varphi^{12 \overline 2}+ \frac{1}{4} (s_x - is_y) \varphi^{\overline 1 \overline 2  2} + \frac{1}{4} (- \overline u _y  + i s +  i  \overline u_x )\varphi^{12 \overline1} + \frac{1}{4} (- u_y- i s -i  u _x ) \varphi^{\overline 1 \overline 2 1}
$$
It follows
$$
\overline \partial \omega =   \frac{1}{4} (s_x - is_y) \varphi^{\overline 1 \overline 2  2} +  \frac{1}{4} (- u_y- i s -i  u _x ) \varphi^{1\overline 1 \overline 2}.
$$
Since $\varphi^{\overline 1 \overline 2  2}$ and  $ \varphi^{1\overline 1 \overline 2}$ are both $\partial$-closed,  we get
 $$ \begin{array}{lcl} \partial  \overline \partial \omega &= &  \frac{1}{4}  \partial (s_x - is_y) \varphi^{\overline 1 \overline 2  2} +  \frac{1}{4}  \partial (- u_y- i s -i  u _x) \varphi^{1\overline 1 \overline 2} \\[2pt]
 & = &   \frac{1}{4}  \partial (s_x - is_y) \varphi^{\overline 1 \overline 2  2}
 \end{array}$$

\begin{pr} Let $\omega = \frac{1}{2}  i  r  (x,y)\,  \varphi^{1 \overline 1} + \frac{1}{2}  i s (x,y) \,  \varphi^{2 \overline 2}  + \frac{1}{2} (u (x,y)\varphi^{1 \overline 2} -  \overline u (x,y)  \varphi^{2 \overline1})$  be the fundamental form of  a  $T^2$-invariant  Hermitian metric  on the Kodaira Thurston surface $M$.

\begin{enumerate}

\item [(a)]  $\omega$   is pluriclosed if and only $s (x,y) = s$ is a constant function.

 \item [(b)] If $\omega$ is pluriclosed, then the Bismut Ricci form has the following expression
\begin{equation} \label{rhobunotzero}
\begin{array}{lcl}
(\rho^B)^{1,1} &=&   \frac{i}{2}  \left (-\frac{1}{2}  \partial^2_x \left(  \log(rs-|u|^2)  \right)
  - \frac{1}{2}  \partial^2_y \left(  \log(rs-|u|^2)  \right)
 - (h_1)_y   - (h_2)_x   -h_3 \right)  \, \varphi^{1 \overline 1} \\[4pt]
&& - (- \frac{1}{4} (h_3)_x + \frac 14 (h_4)_y + \frac i 4 (h_3)_y + \frac i 4 (h_4)_x)  \,  \varphi^{1 \overline 2}\\[4pt]
&&  + (- \frac{1}{4} (h_3)_x + \frac 14 (h_4)_y - \frac i 4 (h_3)_y - \frac i 4 (h_4)_x)  \,   \varphi^{2 \overline 1},
 \end{array}
\end{equation}
where  $u_1$ and $u_2$ are respectively  the real part  and the imaginary  part of $u$ and
$$
\begin{array}{lcl}
h_1 &=& \frac{1} {s}   (u_2  h_3  - u_1 h_4),\\[4pt]
h_2 &=& \frac{1} {s}   (u_1  h_3  + u_2 h_4),\\[4pt]
h_3 &=& \frac{s} {(rs- |u|^2)}    (- s - (u_1)_x - (u_2)_y),\\[4pt]
h_4 &= & \frac{s} {(rs- |u|^2)}  ((u_1)_y - (u_2)_x)\\[2pt]
\end{array}
$$
\end{enumerate}
\end{pr}

\begin{proof}  To prove $(a)$ we use that  $\partial \overline \partial \omega =0$ if and only if $\partial (s_x - i s_y) =0$. Since
 $$
 \partial (s_x - i s_y) = - \frac 12 i (s_{xx} + s_{yy}) \varphi^1
 $$
 we get that the pluriclosed condition is equivalent to $s$ to be constant.

To prove $(b)$ we  will use  \eqref{relRicciforms}.
If  we  write the complex function $u$ as  $u = u_1 + i u_2$ for a $T^2$-invariant  pluriclosed metric we get
$$
d \omega =  (-s - (u_1)_x - (u_2)_y)  \, e^{123} + ((u_1)_y - (u_2)_x) \, e^{124} $$
with $s$ constant.
Let
$$
\theta = h_1 e^1 + h_2 e^2 + h_3 e^3 + h_4 e^4.$$
with $h_i(x,y)$ real functions.  By a direct computation we have
$$
\begin{array}{lcl}
\theta \wedge \omega &= & (s h_1  + u_1 h_4 - h_3 u_2) e^{134} + (u_1 h_1 - u_2 h_2 + r h_4) e^{124}\\[3pt]
&&  (- u_2  h_1  -   u_1 h_2   + r h_3 ) e^{123} + (s h_2-  u_1 h_3 - u_2  h_4) e^{234}.
\end{array}
$$
By imposing
$$
d \omega = \theta \wedge \omega
$$
we obtain the system
$$
\left \{ \begin{array}{l}  - u_2 h_1 - u_1  h_2 + r h_3  = - s - (u_1)_x - (u_2)_y,\\[2pt]
u_1 h_1 - u_2 h_2 + r h_4 = (u_1)_y - (u_2)_x,\\[2pt]
s h_1 - u_2 h_3 + u_1 h_4 =0,\\[2pt]
s h_2 - u_1 h_3 - u_2 h_4 =0
\end{array} \right.
$$
in the variables $h_i$.
Therefore
$$
\begin{array}{lcl}
h_1 &=& \frac{1} {(rs- |u|^2)}  [ u_2  (- s - (u_1)_x - (u_2)_y) - u_1 ((u_1)_y - (u_2)_x) ] = \frac{1}{s} (u_2 h_3 - u_1 h_4),\\[4pt]
h_2 &=& \frac{1} {(rs- |u|^2)}   [ u_1 (- s - (u_1)_x - (u_2)_y) +u_2 ((u_1)_y - (u_2)_x) ] = \frac{1}{s} (u_1 h_3 + u_2 h_4) ,\\[4pt]
h_3 &=& \frac{s} {(rs- |u|^2)}    (- s - (u_1)_x - (u_2)_y),\\[4pt]
h_4 &= & \frac{s} {(rs- |u|^2)}  ((u_1)_y - (u_2)_x)\\[2pt]
\end{array}
$$
and as a consequence
$$
\begin{array}{lcl}
d  (J\theta) &=&  d h_1 \wedge e^2 -  d h_2  \wedge e^1 + d h_3 \wedge  e^4 + h_3     e^1 \wedge e^2 -dh_4  \wedge e^3\\[2pt]
&=& [(h_1)_y  + (h_2)_x  + h_3] \, e^{12}  + (h_3)_x e^{24} + (h_3)_y e^{14}  -  (h_4)_x e^{23} + (h_4)_y e^{13}.
\end{array}
$$
Since
$$
\begin{array}{l}
e^{12} = \frac{i}{2} \varphi^{1 \overline 1},\\[2pt]
e^{13} = \frac{1}{4} (\varphi^{12} + \varphi^{1 \overline 2} - \varphi^{2\overline 1} + \varphi^{\overline 1 \overline 2}),\\[2pt]
e^{14} = - \frac{i}{4} (\varphi^{12} - \varphi^{1 \overline 2} - \varphi^{2\overline 1} - \varphi^{\overline 1 \overline 2}),\\[2pt]
e^{23} = - \frac{i}{4} (\varphi^{12} + \varphi^{1 \overline 2} + \varphi^{2 \overline 1} + \varphi^{\overline 1 \overline 2}),\\[2pt]
e^{24} =  \frac{1}{4} (\varphi^{12} - \varphi^{1 \overline 2} + \varphi^{2\overline 1} + \varphi^{\overline 1 \overline 2}),
\end{array}
$$
we get
$$
\begin{array}{lcl}
(d  (J\theta))^{1,1} &= &\frac{i}{2}  [(h_1)_y   + (h_2)_x  + h_3] \, \varphi^{1 \overline 1} + (- \frac{1}{4} (h_3)_x + \frac 14 (h_4)_y + \frac i 4 (h_3)_y + \frac i 4 (h_4)_x)  \,  \varphi^{1 \overline 2}\\[2pt]
&& -  (- \frac{1}{4} (h_3)_x + \frac 14 (h_4)_y - \frac i 4 (h_3)_y - \frac i 4 (h_4)_x)  \,   \varphi^{2 \overline 1}.
\end{array}
$$


To find the Chern-Ricci form $\rho^C$ we use that it is the curvature of the canonical bundle which gives  the formula
$$\rho^C = -\sqrt{-1}\partial\overline{\partial} \log \det(g_{\alpha\overline{\beta}}) = -\frac{1}{2}dd^c \log \det(g_{\alpha\overline{\beta}}),$$
where $(g_{\alpha\overline{\beta}})$ is the matrix of $g$ is any local holomorphic $(1,0)$-frame. For such frame we can use $\chi^1= \varphi^1 = d(y+ix)$ and $\chi^2 = d(w+iz-\frac{1}{2}x^2)$ since $y+ix$ and $w+iz -\frac{1}{2}x^2$ are holomorphic coordinates on the universal cover of $M$, and the fundamental group is acting completely discontinuous. Since the $(1,0)$-form $\varphi^2 = dw+i(dz-xdy)$ is global, we can see that $$\chi^2 = \varphi^2+ix\varphi^1.$$
If $\tilde{H}$ is the Hermitian matrix of the metric $g$ in the basis $\chi^1,\chi^2$, and $H$ as above is the Hermitian matrix of $g$ in the basis $\varphi^1,  \varphi^2$, by the change of basis formula we get $$\det(\tilde{H}) = \det(H) = rs-|u|^2.$$
In particular, we have
$$\rho^C = -\frac{1}{2} dd^c log(rs-|u|^2).$$
From here and the fact that $r, s, u$ depend only on $x,y$ we get the formula
$$\rho^C = -\frac{1}{2}(\partial_x^2 + \partial_y^2) \log(rs-|u|^2) \, e^{12}.$$
Note that when $u=0$ one gets the formula from   \cite[Lemma 3]{G-G-P}.  From here we get (2).

\end{proof}

\smallskip

Then the pluriclosed flow can be written as the system of PDE's
$$
\left \{ \begin{array}{lcl}
\frac{\partial r} {\partial t} &=&  \frac{1}{2} ( \partial^2_x  +  \partial^2_y)  \left(  \log(rs-|u|^2)  \right)
    +  \frac 12 \partial_y \left(  \frac{r_y}{r} \right)  + (h_1)_y   +(h_2)_x  + h_3,\\[4pt]
\frac{\partial u_1} {\partial  t} &=&    - \frac{1}{2}  (h_3)_x  + \frac 12 (h_4)_y, \\[4pt]
\frac{\partial u_2} {\partial  t} &=&  \frac 12  (h_3)_y   + \frac 12   (h_4)_x, \\[4pt]
\frac{\partial s} {\partial  t}  &= &0,\\[4pt]
 \end{array}
 \right.
$$
where
$$
\begin{array}{lcl}
h_1 &=&  \frac{1}{s} (u_2 h_3 - u_1 h_4),\\[4pt]
h_2 &=&  \frac{1}{s} (u_1 h_3 + u_2 h_4) ,\\[4pt]
h_3 &=& \frac{s} {(rs- |u|^2)}    (- s - (u_1)_x - (u_2)_y),\\[4pt]
h_4 &= & \frac{s} {(rs- |u|^2)}  ((u_1)_y - (u_2)_x).\\[2pt]
\end{array}
$$
By a direct computation
$$
\begin{array}{lcl}
 (h_1)_y  +(h_2)_x + h_3& = &   - \frac{1}{s^2} (h_3^2 + h_4^2) (rs - |u|^2)  + \frac{1}{s} u_1 \left (- (h_3)_x   + (h_4)_y \right )\\[4pt]
 &&   + \frac{1}{s} u_2  \left ( (h_3)_y  +  (h_4)_x \right )\\[4pt]
 &=&  - \frac{1}{s^2} (h_3^2 + h_4^2) (rs - |u|^2) + \frac {1}{s}  \partial_t (|u|^2)
 \end{array}
 $$
 and so the system reduces to
\begin{equation} \label{NewSystem}
\left \{ \begin{array}{lcl}
\frac{\partial r} {\partial t} &=&  \frac{1}{2}   (\partial^2_x  + \partial^2_y ) \left(  \log(rs-|u|^2)  \right)
  + \frac{1}{s^2} (h_3^2 + h_4^2) (rs - |u|^2)  + \frac 1 s  \partial_t (|u|^2))),\\[4pt]
\frac{\partial u_1} {\partial  t} &=& - \frac{1}{2}  (h_3)_x + \frac 12 (h_4)_y, \\[4pt]
\frac{\partial u_2} {\partial  t} &=&  \frac 12  (h_3)_y  + \frac 12   (h_4)_x, \\[4pt]
\frac{\partial s} {\partial  t}  &= &0,\\[4pt]
h_3 &=& \frac{s} {(rs- |u|^2)}    (- s - (u_1)_x - (u_2)_y),\\[4pt]
h_4 &= & \frac{s} {(rs- |u|^2)}  ((u_1)_y - (u_2)_x).\\[2pt]
 \end{array}
 \right.
\end{equation}

The  previous system is  parabolic and quasilinear and applying the result for instance in \cite{MM} it has short time existence.

\begin{remark}  \begin{enumerate} \item Note that we  get that $\frac{\partial u_1} {\partial  t} =  0 = \frac{\partial u_2} {\partial  t}$ if and only if $h_3 + i h_4$ is a holomorphic function on $T^2$ and so a constant function.

 \item If $r, s$ and $u$ are constant one gets
 $$
\left \{ \begin{array}{lcl}
\frac{\partial r} {\partial t} &=&  \frac{s^2}{ (rs - |u|^2)},\\[4pt]
\frac{\partial u_1} {\partial  t} &=& 0, \\[4pt]
\frac{\partial u_2} {\partial  t} &=& 0. \\[4pt]
 \end{array}
 \right.
$$

\item The pluriclosed flow on compact complex surfaces which are  the total space of a holomorphic $T^2$-principal bundle over a Riemann surface $\Sigma$ have been studied in   \cite{Streets},
showing  that the solution to
pluriclosed flow with initial data   a $T^2$-invariant metric $\omega_0$  with $u=0$
 exists on $[0,  + \infty)$, and  that  $(M,  \omega(t))$  converges in the Gromov-Hausdorff topology to a point.

\end{enumerate}
\end{remark}

\smallskip

\begin{te}  Let $\omega = \frac{1}{2}  i  r  (x,y)\,  \varphi^{1 \overline 1} + \frac{1}{2}  i s \,  \varphi^{2 \overline 2}  + \frac{1}{2} (u (x,y)\varphi^{1 \overline 2} -  \overline u (x,y)  \varphi^{2 \overline1})$  be  the fundamental form of a   $T^2$-invariant pluriclosed   metric  $g$ on the Kodaira-Thurston surface $M$. Then  $\omega$ is Vaisman  if and only if  the functions
$$
\begin{array}{lcl}
h_3 &=& \frac{s} {(rs- |u|^2)}    (- s - (u_1)_x - (u_2)_y),\\[4pt]
h_4 &= & \frac{s} {(rs- |u|^2)}  ((u_1)_y - (u_2)_x)\\[2pt]
\end{array}
$$
are  both constant. Moreover, if $\omega$ is Vaisman,  the Lee vector field  $T$ is given by  $T = \frac{h_3}{s} e_3 + \frac{h_4}{s} e_4$ and
$$
d (J \theta) = - \frac{1}{s^2} (h_3^2 + h_4^2) (rs - |u|^2) e^{12}.
$$
\end{te}

\begin{proof}  The $T^2$-invariant pluriclosed metric $g$ is locally conformally K\"ahler if and only if $d \theta =0$. The vanishing of
$$
d \theta = d(h_1) \wedge e^1 + d(h_2) \wedge e^2 + d(h_3) \wedge e^3 + d(h_4) \wedge e^4 + h_4 e^{12}
$$
 is equivalent to the conditions
$$
d h_3 = dh_4 =0,  \quad  - (h_1)_x +  (h_2)_y + h_4 =0.
$$
By using that $h_3$ and $h_4$ are both constant and the expressions
$$
\begin{array}{lcl}
h_1 &=& \frac{1} {s}   (u_2  h_3  - u_1 h_4),\\[4pt]
h_2 &=& \frac{1} {s}   (u_1  h_3  + u_2 h_4),\\[4pt]
\end{array}
$$
we have that  the condition  $  - (h_1)_x +  (h_2)_y + h_4 =0$ is always satisfied.
Now  we  only need to prove that  if $h_3$ and $h_4$ are both constant then  the Lee vector field $T$ is holomorphic,  since automatically $g$ will be Vaisman.
$T$ is the metric dual of $\theta$, so we must have  $g (T, X) = \theta (X)$, for every vector field $X$. By imposing   $g (T, e_i) = \theta (e_i) = h_i$, for every $i = 1, \ldots, 4$,  we have that the Lee vector field $T$  is given by
$$
\begin{array}{lcl}
T &= &\frac{(h_1 s - h_3 u_2  + h_4 u_1)}{rs - |u|^2} e_1 +  \frac{(h_2 s - h_3 u_1- h_4 u_2)}{rs - |u|^2} e_2 \\[5pt]
&&+ \frac{(-h_1 u_2 - h_2 u_1 +  h_3 r)}{rs - |u|^2} e_3  + \frac{(h_1 u_1 -  h_2 u_2 + h_4 r)}{rs - |u|^2}e_4.
\end{array}
$$
Using that
$$
\begin{array}{lcl}
h_1 &=& \frac{1} {s}   (u_2  h_3  - u_1 h_4),\\[4pt]
h_2 &=& \frac{1} {s}   (u_1  h_3  + u_2 h_4),
\end{array}
$$
it follows that  $T = \frac{h_3}{s} e_3 + \frac{h_4}{s} e_4$. Therefore $T$ is holomorphic since $[T, JX] = J[T, X]$, for every vector field $X$.
The last part  of the theorem follows from
$$
d  (J\theta)  =  [(h_1)_y  + (h_2)_x  + h_3] \, e^{12}
$$
and
$$
(h_1)_y   +(h_2)_x + h_3 = - \frac{1}{s^2} (h_3^2 + h_4^2) (rs - |u|^2).
$$

\end{proof}

\smallskip

As a consequence of the previous theorem we can prove

\begin{te}Let  $\omega_0 $ be  the fundamental form of a    $T^2$-invariant  Vaisman metric on the Kodaira-Thurston surface $M$, then the pluriclosed flow starting with $\omega_0$  preserves the Vaisman condition.
\end{te}

\begin{proof}

Let $\omega_0 = \frac{1}{2}  i  r (x,y)\,  \varphi^{1 \overline 1} + \frac{1}{2}  i s \,  \varphi^{2 \overline 2}  + \frac{1}{2} (u (x,y)\varphi^{1 \overline 2} -  \overline u (x,y)  \varphi^{2 \overline1})$ be the fundamental form of the    $T^2$-invariant  Vaisman metric.
Then
$$
\begin{array}{lcl}
h_3 &=& \frac{s} {(rs- |u|^2)}    (- s - (u_1)_x - (u_2)_y),\\[4pt]
h_4 &= & \frac{s} {(rs- |u|^2)}  ((u_1)_y - (u_2)_x)\\[2pt]
\end{array}
$$
are both constants.
To  prove that the Vaisman condition is preserved under the  pluriclosed flow,  we   use  the following ansatz
 $$
 \omega(t) =   \frac{1}{2}  i   \tilde r (x,y,t)\,  \varphi^{1 \overline 1} + \frac{1}{2}  i s \,  \varphi^{2 \overline 2}  + \frac{1}{2} (u (x,y)\varphi^{1 \overline 2} -  \overline u (x,y)  \varphi^{2 \overline1}),
 $$
with $\tilde r (x,y, 0) = r(x,y)$ and $s$ constant and use that  the equation
$$
\partial_t  \tilde r =  \frac{1}{2}  \partial_x \left(  \frac{\tilde r_x}{\tilde r}  \right )  +  \frac 12 \partial_y \left(  \frac{ \tilde r_y}{\tilde r} \right) + \frac{1}{s^2} (h_3^2 + h_4^2) (\tilde rs - |u|^2),
$$
is   quasi-linear parabolic  and so it  admits a solution.
Indeed, the equations in \eqref{NewSystem} reduce only  to the previous equation, since $\frac{\partial u_1} {\partial  t} = \frac{\partial u_2} {\partial  t} =0$.
\end{proof}

\smallskip

Note that, using deformartions of type II, we can show the existence of Vaisman $T^2$-invariant metrics  with non-constant scalar curvature  on the Kodaira-Thurston surface.  Start with the Vaisman metric  with fundamental form  $\omega = e^{12} + e^{34}$, we have that $\theta = - e^3$. If we apply the deformation of type II  in Section 2 using a   function $f = f(x,y)$ such that $1 - f_{xx} - f_{yy} > 0$, we obtain  that the  metric with fundamental form
$$
\begin{array}{lcl}
\tilde \omega &= &  e^{12} + e^{34}  - e^3  \wedge J df - df \wedge  e^4 + df \wedge J df - d d^c f\\[3pt]
& = &(1 + (f_x)^2 + (f_y)^2 - f_{xx} - f_{yy}) e^{12} + e^{34} - f_x (e^{13} + e^{24}) - f_y (e^{14} - e^{23})
\end{array}
$$
is Vaisman with $\tilde \theta = - e^3 + df = - e^3 + f_y e^1 + f_x e^2$.  So
$$
d (J  \tilde \theta) = d (- e^4 + f_y e^2 - f_x e^1) =  (-1 + f_{xx} +  f_{yy}) e^{12}.
$$
Moreover, $$
\begin{array}{lcl}
\tilde {\rho}^{Ch}  &= &     -\frac{1}{2}(\partial_x^2 + \partial_y^2) \log(rs-|u|^2)  \,  e^{12}\\[5pt]
&  = &-  \frac{1}{2}  \frac{ 1} { (-1 + f_{xx} +  f_{yy})}  \left(  ( \partial_x^2 + \partial_y^2) \log(rs-|u|^2)  \right)     d (J  \tilde \theta),
\end{array}
$$
where $r = 1 + (f_x)^2 + (f_y)^2 - f_{xx} - f_{yy}, s = 1$ and $u = - f_x - i f_y$. Therefore
$$
\tilde {\rho}^{Ch}  =  -  \frac{1}{2}  \frac{ 1} { (-1 + f_{xx} +  f_{yy})}  \left(  (\partial_x^2 + \partial_y^2) \log( 1 - f_{xx} - f_{yy}) \right)    d (J  \tilde \theta).
$$

Non-constant functions $f$ such that $f_{xx} +  f_{yy}  <1$  exist. Therefore  if  $$\frac{ 1} { (-1 + f_{xx} +  f_{yy})} \left(  (\partial_x^2 + \partial_y^2) \log( 1 - f_{xx} - f_{yy}) \right) $$  is non-constant,  the scalar curvature of $\tilde \omega$ is non-constant.

\smallskip

{\bf Acknowledgements.}
Anna Fino is partially supported by Project PRIN 2017 “Real and complex manifolds: Topology, Geometry and Holomorphic Dynamics”, by GNSAGA (Indam) and by a grant from the Simons Foundation (\#944448). Gueo Grantcharov is partially supported by a grant from the Simons Foundation (\#853269). We would like to thank  Liviu Ornea,  Jeff Streets and Luigi Vezzoni for  useful comments.

\end{document}